\documentclass[final,onefignum,onetabnum]{siamonline190516}

\usepackage{braket,amsfonts}
\usepackage{amsmath,amssymb,amsxtra}
\usepackage{bm,mathrsfs,marvosym,eurosym}
\usepackage{mathtools,enumerate}
\usepackage{bbm}
\usepackage{array}
\usepackage[normalem]{ulem}

\usepackage[caption=false]{subfig}

\newsiamthm{claim}{Claim}
\newsiamremark{remark}{Remark}
\newsiamremark{hypothesis}{Hypothesis}
\crefname{hypothesis}{Hypothesis}{Hypotheses}
\newsiamthm{problem}{Problem}
\newsiamthm{assumption}{Assumption}

\numberwithin{figure}{section}
\numberwithin{table}{section}

\usepackage{algorithmic}

\usepackage{graphicx}

\usepackage{pstricks,pst-plot,pst-func}
\usepackage{pspicture}
\usepackage{curves}

\Crefname{ALC@unique}{Line}{Lines}

\include{rgb}

\usepackage{amsopn}

\def\d{{\mathrm d}}

\def\G{{\mathcal G}}
\def\F{{\mathcal F}}

\def\R{{\mathbb R}}

\def\AA{\text{${\mathcal O}\hskip-3pt\iota$}}

\def\<{{\langle }}
\def\>{{\rangle }}

\usepackage{xspace}
\usepackage{bold-extra}
\usepackage[most]{tcolorbox}

\colorlet{texcscolor}{blue!50!black}
\colorlet{texemcolor}{red!70!black}
\colorlet{texpreamble}{red!70!black}
\colorlet{codebackground}{black!25!white!25}

\lstdefinestyle{siamlatex}{%
  style=tcblatex,
  texcsstyle=*\color{texcscolor},
  texcsstyle=[2]\color{texemcolor},
  keywordstyle=[2]\color{texemcolor},
  moretexcs={cref,Cref,maketitle,mathcal,text,headers,email,url},
}

\tcbset{%
  colframe=black!75!white!75,
  coltitle=white,
  colback=codebackground, 
  colbacklower=white, 
  fonttitle=\bfseries,
  arc=0pt,outer arc=0pt,
  top=1pt,bottom=1pt,left=1mm,right=1mm,middle=1mm,boxsep=1mm,
  leftrule=0.3mm,rightrule=0.3mm,toprule=0.3mm,bottomrule=0.3mm,
  listing options={style=siamlatex}
}

\newtcblisting[use counter=example]{example}[2][]{%
  title={Example~\thetcbcounter: #2},#1}

\newtcbinputlisting[use counter=example]{\examplefile}[3][]{%
  title={Example~\thetcbcounter: #2},listing file={#3},#1}

\DeclareTotalTCBox{\code}{ v O{} }
{ 
  fontupper=\ttfamily\color{black},
  nobeforeafter,
  tcbox raise base,
  colback=codebackground,colframe=white,
  top=0pt,bottom=0pt,left=0mm,right=0mm,
  leftrule=0pt,rightrule=0pt,toprule=0mm,bottomrule=0mm,
  boxsep=0.5mm,
  #2}{#1}

\patchcmd\newpage{\vfil}{}{}{}
\begin{tcbverbatimwrite}{tmp_\jobname_header.tex}
\title{Robust Convergence of Parareal Algorithms with Arbitrarily High-order Fine Propagators
\thanks{\funding{The work of J.\ Yang is supported by National Natural Science Foundation of China (NSFC) Grant No. 11871264
and the research of Z.\ Yuan and Z.\ Zhou is partially supported by Hong Kong RGC grant (No. 15304420).}}}

\author{Jiang Yang\thanks{Department of Mathematics \& SUSTech International Center for Mathematics,
Southern University of Science and Technology, Shenzhen 518055, China.
(\texttt{yangj7sustech.edu.cn})}
\and Zhaoming Yuan\thanks{Department of Applied Mathematics, The Hong Kong Polytechnic University, Kowloon, Hong Kong. Department of Mathematics \& SUSTech International Center for Mathematics,
Southern University of Science and Technology, Shenzhen 518055, China.
(\texttt{zhaoming.yuan@connect.polyu.hk})}
\and Zhi Zhou\thanks{Department of Applied Mathematics, The Hong Kong Polytechnic University, Kowloon, Hong Kong.
(\texttt{zhizhou@polyu.edu.hk, zhizhou0125@gmail.com})}}
\date{\today}

\headers{Parareal Algorithms with Arbitrarily High-order Fine Propagators}{Jiang Yang, Zhaoming Yuan and Zhi Zhou}
\end{tcbverbatimwrite}
\title{Robust Convergence of Parareal Algorithms with Arbitrarily High-order Fine Propagators
\thanks{\funding{The work of J.\ Yang is supported by National Natural Science Foundation of China (NSFC) Grant No. 11871264
and the research of Z.\ Yuan and Z.\ Zhou is partially supported by Hong Kong RGC grant (No. 15304420).}}}

\author{Jiang Yang\thanks{Department of Mathematics \& SUSTech International Center for Mathematics,
Southern University of Science and Technology, Shenzhen 518055, China.
(\texttt{yangj7sustech.edu.cn})}
\and Zhaoming Yuan\thanks{Department of Applied Mathematics, The Hong Kong Polytechnic University, Kowloon, Hong Kong. Department of Mathematics \& SUSTech International Center for Mathematics,
Southern University of Science and Technology, Shenzhen 518055, China.
(\texttt{zhaoming.yuan@connect.polyu.hk})}
\and Zhi Zhou\thanks{Department of Applied Mathematics, The Hong Kong Polytechnic University, Kowloon, Hong Kong.
(\texttt{zhizhou@polyu.edu.hk, zhizhou0125@gmail.com})}}
\date{\today}

\headers{Parareal Algorithms with Arbitrarily High-order Fine Propagators}{Jiang Yang, Zhaoming Yuan and Zhi Zhou}

\begin{document}
\maketitle

\begin{tcbverbatimwrite}{tmp_\jobname_abstract.tex}
\begin{abstract}
The aim of this paper is to analyze the robust convergence of a class of parareal algorithms for solving
parabolic problems. The coarse propagator is fixed to the backward Euler method
and the fine propagator is a high-order single step integrator.
Under some conditions on the fine propagator,
we show that there exists some critical $J_*$ such
that the parareal solver converges linearly with a convergence rate near $0.3$,
provided that the ratio between the coarse time step and fine
time step named $J$ satisfies $J \ge J_*$.
The convergence is robust even if the problem data is nonsmooth and incompatible with boundary conditions. The qualified methods include all absolutely stable single step methods,
whose stability function satisfies $|r(-\infty)|<1$,
and hence the fine propagator could be arbitrarily high-order. Moreover, we examine some popular high-order single step methods, e.g., two-, three- and four-stage Lobatto IIIC methods,
and verify that the corresponding parareal algorithms converge linearly with a factor $0.31$ and the threshold for these cases is $J_* = 2$.
Intensive numerical examples are presented to support and complete our theoretical predictions.
\end{abstract}

\begin{keywords}
parareal algorithm, parabolic problems, arbitrarily high-order,
single step integrator, convergence factor
\end{keywords}

\begin{AMS}
Primary 65M12, 65M60; Secondary 65L06
\end{AMS}
\end{tcbverbatimwrite}
\begin{abstract}
The aim of this paper is to analyze the robust convergence of a class of parareal algorithms for solving
parabolic problems. The coarse propagator is fixed to the backward Euler method
and the fine propagator is a high-order single step integrator.
Under some conditions on the fine propagator,
we show that there exists some critical $J_*$ such
that the parareal solver converges linearly with a convergence rate near $0.3$,
provided that the ratio between the coarse time step and fine
time step named $J$ satisfies $J \ge J_*$.
The convergence is robust even if the problem data is nonsmooth and incompatible with boundary conditions. The qualified methods include all absolutely stable single step methods,
whose stability function satisfies $|r(-\infty)|<1$,
and hence the fine propagator could be arbitrarily high-order. Moreover, we examine some popular high-order single step methods, e.g., two-, three- and four-stage Lobatto IIIC methods,
and verify that the corresponding parareal algorithms converge linearly with a factor $0.31$ and the threshold for these cases is $J_* = 2$.
Intensive numerical examples are presented to support and complete our theoretical predictions.
\end{abstract}

\begin{keywords}
parareal algorithm, parabolic problems, arbitrarily high-order,
single step integrator, convergence factor
\end{keywords}

\begin{AMS}
Primary 65M12, 65M60; Secondary 65L06
\end{AMS}






\section{Introduction}\label{sec:intro}
The main focus of this paper is to study the convergence of a class of parareal solver for the parabolic problems. Specifically,
we let $T >0, u^0\in H,$ and consider the initial value
problem of seeking $u \in C((0,T];D(A))\cap C([0,T];H)$ satisfying
\begin{equation}\label{eqn:pde}
\left \{
\begin{aligned}
&u' (t) + Au(t) = f(t), \quad 0<t<T,\\
& u(0)=u^0 ,
\end{aligned}
\right .
\end{equation}
where $A$ is a positive definite, selfadjoint, linear operator with a compact inverse, defined in
Hilbert space $(H, (\cdot , \cdot )) $ with domain $D(A)$ dense in $H$.
Here $u^0 \in H$ is a given initial condition
and $f  : [0,T] \to H$ is a given forcing term.
Throughout the paper, $\| \cdot \|$ denotes the norm of the space $H$.

Parallel-in-time (PinT) methods, dating back to the work of Nievergelt in
1964 \cite{Nievergelt:1964}, have attracted a lot of interest in the last several decades.
The parareal method, introduced in 2001 \cite{LionsMadayTurinici:2001}, is perhaps one of the most popular PinT algorithms.
This method is relatively simple to implement, and can be employed for any single step integrators.
In recent years, the parareal algorithm and some relevant algorithms, have been applied in many fields,
such as turbulent plasma \cite{Reynolds-Barredo:2012, Reynolds-Barredo:2013}, structural (fluid) dynamics \cite{CortialFarhat:2009, Farhat:2003},
molecular dynamics \cite{Baffico:2002}, optimal control \cite{MadaySalomonTurinici:2007, MathewSarkisSchaerer:2010}
Volterra integral equations and fractional models \cite{LiTangXu:2013, WuZhou:2017frac}, etc.
We refer the interested reader to
survey papers \cite{Gander:2015, OngSchroder:2020} and references
therein.

The parareal algorithm is defined by
using two time propagators, $\mathcal{G}$ and $\mathcal{F}$,  associated with the large step size $\Delta T$ and the small step size $\Delta t$ respectively, where we assume that the ratio $J = \Delta T/\Delta t$ is an integer greater than $1$. The fine time propagator $\mathcal{F}$ is operated with small step size $\Delta t$ in each coarse sub-interval parallelly, after which the coarse time propagator $\mathcal{G}$ is operated with
large step size $\Delta T$ sequentially for corrections. In general, the coarse propagator  $\mathcal{G}$  is assumed to be much cheaper than the fine propagator $\mathcal{F}$.
Therefore, throughout this paper, we fix $\mathcal{G}$ to the backward-Euler method and study the choices of $\mathcal{F}$.
Then a natural question  arises related to convergence of the parareal algorithm.  For parabolic type problems,
in the pioneer work \cite{Bal:2005}, Bal proved a fast convergence of the parareal method with a strongly stable coarse propagator
and the exact fine propagator,
provided some regularity assumptions on the problem data. The analysis works for both linear and nonlinear problems.
This convergence behavior
is clearly observed in numerical experiments,
see e.g. Figure \ref{fig:ex1-a}.
However, without those regularity assumptions, the convergence observed from the empirical experiments will be much slower than expected, cf. Figure \ref{fig:ex1-b}.

This interesting phenomenon motivates the current work, where we aim to study the convergence of parareal algorithm which is expected to be robust in the case of nonsmooth / incompatible problem data, that is related to various applications, e.g., optimal control, inverse problems, and stochastic models.
There have existed some case studies.
In \cite{MathewSarkisSchaerer:2010},
Mathew, Sarkis and Schaerer considered
the backward Euler method as the fine propagator  and proved
the robust convergence of the parareal algorithm with a convergence factor $0.298$ (for all $J \ge 2$); see also \cite{GanderVandewalle:2007,Staff:2005} for some related discussion. In \cite{Wu:2015}, Wu showed that the convergence factors
for the second-order diagonal implicit Runge--Kutta method
and a single step TR/BDF2 method (i.e., the ode23tb solver for ODEs in MATLAB)
are $0.316$ (with $J\ge 2$) and $0.333$ (with $J\ge 2$), respectively.
These error bounds might be slightly improved
by increasing $J_*$.
See also \cite{WuZhou:2015} for the analysis for a third-order diagonal implicit Runge--Kutta method with a convergence factor $0.333$ ($J_*=4$). For fourth-order Gauss--Runge--Kutta integrator, in \cite{WuZhou:2015} Wu and Zhou showed that the threshold depends on both the largest eigenvalue of operator $A$ and the step size $\Delta t$. Note that the eigenvalues of $A$ may approach infinity, e.g. $A=-\Delta$ with homogeneous boundary conditions. Therefore, this kind of integrators might not be suitable for the parareal algorithm.

Then a natural question arises: \textsl{\uline{in what case there exists a threshold $J_*>0$ (independent of step sizes $\Delta T$, $\Delta t$, terminal time $T$,
problem data $u^0$ and $f$,
as well as the distribution of spectrum of the elliptic
operator $A$), such that
for any $J\ge J_*$, the parareal algorithm for solving the parabolic equation \eqref{eqn:pde} converges robustly}}?
Our study provides a positive answer to this question: if the fine propagator is strongly stable, in sense that the stability function satisfies $|r(-\infty)| \in [0,1)$, then
there must exist such a positive threshold $J_*$ so that for all $J\ge J_*$
the parareal algorithm converges linearly with convergence factor close to $0.3$.
The convergence is robust even if the initial data is nonsmooth or incompatible with boundary conditions.
Noting that all L-stable Runge--Kutta schemes satisfy that condition, so the fine propagator  can be arbitrarily high-order.
As examples, we analyzed three popular L-stable schemes, i.e., two-, three-, four-stage Lobatto IIIC schemes. We show that for all these cases the parareal algorithm converges linearly with factor less than $0.31$
and $J_*=2$.
Our theoretical results are fully supported by numerical experiments.

The rest of the paper is organized as follows. In Section \ref{sec:parareal}, we introduce singe step integrators
and parareal algorithms for solving the parabolic problem.
Then we show the convergence of the algorithm in Section  \ref{sec:convergence}
by using the spectrum decomposition.
Moreover, in Section \ref{sec:RK}, we present case studies on three popular L-stable Runge--Kutta schemes, and show a sharper estimate for
the threshold $J_*$. Finally, in Section \ref{sec:numerics}, we present some numerical results to illustrate and complement the theoretical analysis.

\section{Single step methods and parareal algorithm}\label{sec:parareal}
In this section, we present the basic setting of the single step time stepping methods for solving the parabolic equation \eqref{eqn:pde}
and the parareal algorithm. See more detailed discussion in the monograph \cite[Chapter 7-9]{Thomee:2006}
and the comprehensive survey paper \cite{Gander:2015}.

\subsection{Single step integrators for solving parabolic equations}
To begin with, we consider the time discretization for the parabolic equation \eqref{eqn:pde}.
We split the interval $(0,T)$ into $N$ subintervals with the uniform mesh size ${\Delta t}=T/N$,
and set $t_n = n{\Delta t}$, $n=0,1,\ldots,N$.
Then a framework of a single step scheme approximating $u(t_n)$ reads:
\begin{equation}\label{eqn:semi}
  u^{n+1} = r(-{\Delta t} A) u^{n} + {\Delta t} \sum_{i=1}^m p_i(-{\Delta t} A)
  f(t_n+c_i\Delta t) ,\quad \text{for all}~0\le n\le N-1,
\end{equation}
Here, $r(\lambda)$ and $\{p_i(\lambda)\}^m_{i=1}$ are rational functions and $c_i$ are distinct real numbers in $[0,1]$.
Throughout the paper, we assume that the scheme \eqref{eqn:semi} satisfies the following assumptions.\vskip10pt

\begin{itemize}
\item[{\bf (P1)}] $| r (-\lambda)|< 1$ and $|p_i(-\lambda)|\le c$, for all $i=1,\ldots,m$,
uniformly in ${\Delta t}$ and $\lambda > 0$. Besides, the numerator of $p_i(\lambda)$
is of lower degree than its denominator.
\item[{\bf (P2)}] The time stepping scheme \eqref{eqn:semi}  is \textsl{accurate} of order $q$ in sense that
$$  r (-\lambda) = e^{-\lambda} + O(\lambda^{q+1}),\quad \text{as}~\lambda\rightarrow0. $$
and for $0\le j\le q$
$$ \sum_{i=1}^m c_i^j p_i(-\lambda) - \frac{j!}{(-\lambda)^{j+1}}\Big(e^{-\lambda}
- \sum_{\ell=0}^j \frac{(-\lambda)^\ell}{\ell!}\Big) = O(\lambda^{q-j}),\quad \text{as}~\lambda\rightarrow0. $$
\item[{\bf (P3)}] The rational function $r(\lambda)$ is \textsl{strongly stable} in sense that $|r(-\infty)|<1$.
\end{itemize}\vskip5pt

\begin{remark}
Condition (P3) is essential for the convergence of parareal iteration.
If $|r(\infty)| = 1$, e.g., Crank-Nicolson method and implicit Runge-Kutta methods of Gauss type,
the parareal method converges only if the eigenvalues of $A$ is bounded from above
(which is not true for parabolic equations)
and the ratio between the coarse step size and the fine step size is sufficiently large (depending on the upper bound of eigenvalues of $A$). Besides, this condition is also important in case that problem data is nonsmooth, e.g., $u^0\in H$. Time stepping schemes violating this condition may lose the optimal convergence rate in the nonsmooth data case \cite[Chapter 8]{Thomee:2006}.
\end{remark}

Practically, it is convenient to choose $p_i(\lambda)$ that share the same denominator of $ r (\lambda)$:
$$ r (\lambda) = \frac{a_0(\lambda)}{g(\lambda)},\quad \text{and} \quad p_i(\lambda) = \frac{a_i(\lambda)}{g(\lambda)},\quad \text{for}~i=1,2,\ldots,m,$$
where $a_i(\lambda)$ and $g(\lambda)$ are polynomials.
Then the integrator \eqref{eqn:semi} could be written as
\begin{equation*}
g(-{\Delta t} A) u^{n+1} = a_0(-{\Delta t} A) u^{n} + {\Delta t} \sum_{i=1}^m a_i(-{\Delta t} A) f(t_n+c_i\Delta t) ,\quad \text{for all}~1 \le n\le N.
\end{equation*}
See e.g. \cite[pp. 131]{Thomee:2006} for the construction of such rational functions satisfying (P1)-(P3).\vskip5pt

Under those conditions, there holds the following error estimate for the time stepping scheme \eqref{eqn:semi}.
The proof is given in \cite[Theorems 7.2 and 8.3]{Thomee:2006}.
\begin{lemma}\label{lem:conv-00}
Suppose that the Conditions (P1)-(P3) are fulfilled.
Let $u(t)$ be the solution to parabolic equation \eqref{eqn:pde}, and $u^n$
be the solution to the time stepping scheme \eqref{eqn:semi}. Then there holds
\begin{equation*}
\|u^n - u(t_n)\|\leq
c\, (\Delta t)^q \Big( t_n^{-q} \| u^0 \|
+ t_n \sum_{\ell=0}^{q-1} \sup_{s\le t_n} \| A^{q-\ell}f^{(\ell)}(s)  \| + \int_0^{t_n}
\| f^{(q)}(s) \|\,\d s\Big),
\end{equation*}
provided that $v \in H$, $f^{(\ell)}\in C([0,T];\text{Dom}(A^{q-\ell})$ with $0\le \ell\le q-1$ and $f^{(q)} \in L^1(0,T; H)$ for all $\ell<q$.
\end{lemma}

\begin{remark}\label{rem:conv-00}
Lemma \ref{lem:conv-00} indicates that, under Conditions (P1)-(P3), the
solution of the time stepping scheme \eqref{eqn:semi} converges to the exact solution with order $q$
provided that the source term $f$ and initial condition $u_0$
satisfy certain compatibility conditions.
For example, if we consider the parabolic equation where  $A=-\Delta$ with homogeneous Dirichlet boundary condition,
it requires
$(-\Delta)^{\ell} f^{(q-\ell)}= 0$ on the boundary $\partial\Omega$ for $0\le \ell\le q$.
In order to avoid the restrictive compatibility conditions, we shall assume that the time discretization scheme \eqref{eqn:semi} is \textsl{strictly accurate} of order $q$ in sense that
$$ \sum_{i=1}^m c_i^j p_i(-\lambda) - \frac{j!}{(-\lambda)^{j+1}}\Big( r (-\lambda)
- \sum_{\ell=0}^j \frac{(-\lambda)^\ell}{\ell!}\Big) =0,\quad\text{for all}~ 0\le j\le q-1. $$
It is well-known that a single step method
with a given $m \in \mathbb{Z}^+$ could be accurate of order $2m$ (Gauss--Legendre method) \cite[Section 2.2]{Ehle:thesis},
but at most strictly accurate of order $m+1$ \cite[Lemma 5]{BrennerCrouzeixThomee:1982}.
\end{remark}

\begin{remark}\label{rem:conv-01}
The error estimate in Lemma \ref{lem:conv-00} could be slightly improved if the time integrator is L-stable, i.e. $r(-\infty) =0$; see e.g., \cite[Theorem 7.2]{Thomee:2006}.
\end{remark}








\subsection{Parareal algorithm}
Next, we state the parareal solver for the single step scheme \eqref{eqn:semi}.
Let $\Delta T=J \Delta t$, with a positive integer $J\ge2$, be the coarse step size.
Without loss of generality, we assume that $N_c = T/\Delta T$ is an integer, and let $T_n = n\Delta T$
Then, two numerical propagators $\G$ and $\F$ are assigned to the coarse and fine time grids,
where $\G$ is usually a low-order and inexpensive numerical method (such as backward Euler scheme), and $\F$
is given by the single step integrator \eqref{eqn:semi}. Specifically, for $v\in H$ and $f\in C([0,T];H)$, letting
$I$ denote the identity operator,
we define the coarse
and finer propagator as
\begin{equation*}
\G(T_n, \Delta T, v, f) = (I + \Delta T A)^{-1} (v + \Delta T f(T_n)).
\end{equation*}
and
\begin{equation*}
\F(t_n, \Delta t, v, f) =   r(-{\Delta t} A) v + {\Delta t} \sum_{i=1}^m p_i(-{\Delta t} A)  f(t_{n}+c_i\Delta t).
\end{equation*}
respectively. Then, the parareal solver is described in Algorithm \ref{alg:para}.

\begin{algorithm}[ht]
  \center
\caption{Parareal solver for the single step scheme \eqref{eqn:semi}.}
  \begin{algorithmic}[1]\label{alg:para}
    \STATE \textbf{Initialization}:
    Compute
    $U_0^{n+1} = \G(T_n,\Delta T, U_0^{n+1}, f)$ with $U_0^0 =u^0$, $n =0,1,...,N_c - 1$;
    \FOR {$k=0,1,\ldots,K$}
    \STATE \textbf{Step 1}: On each subinterval $[T_n,T_n+1]$, sequetially compute for $j=1,2,\dots,J-1$
    $$  \widetilde U^{n, j+1} = \F(T_n+j \Delta t, \Delta t, \widetilde U^{n, j},f), $$
    with initial value $ \widetilde U^{n, 0} = U_k^n$.
    Let $\widetilde U^{n+1} =  \widetilde U^{n, J}$.\vskip5pt
    \STATE \textbf{Step 2}: Perform sequential corrections: find $U_{k+1}^{n+1}$
    $$  U_{k+1}^{n+1} = \G(T_n, \Delta T, U_{k+1}^{n}, f)  + \widetilde U^{n+1} -   \G(T_n, \Delta T, U_{k}^{n}, f) $$
   with $U_{k+1}^0 = u^0$, for $n =  0,1,...,N_c - 1$;\vskip5pt
    \STATE \textbf{Step 3}:  If $\{ U_{k+1}^n \}_{n=1}^{N_c}$ satisfies the stopping criterion, terminate the iteration; otherwise go to {\bf Step 1}.
    \ENDFOR
  \end{algorithmic}
\end{algorithm}

The aim of this paper is to show that the iterative solution $U_k^n$, generated by the parareal algorithm,
linearly converges to the exact time stepping solution $u^{nJ}$ of the single step integrator \eqref{eqn:semi} with fine time step $\Delta t$, i.e.,
\begin{equation}\label{eqn:conv-1}
\max_{1 \le n\le N}\|  U_{k}^n - u^{nJ}  \| \le  c \,\gamma^k,
\end{equation}
with some convergence factor $\gamma$ strictly smaller than $1$. We shall prove that
 there exists a positive threshold $J_*$,
independent of $\Delta T$, $\Delta t$ and the upper bound of spectrum of $A$, such that
if $J\ge J_*$, then \eqref{eqn:conv-1} is true with $\gamma$ close to $0.3$, under conditions (P1)-(P3).

\section{Convergence analysis}\label{sec:convergence}

Next, we briefly test the convergence factor of parareal iteration.
Taking comparision with the exact time stepping solution in \eqref{eqn:semi}, we arrive at
\begin{equation*}
\begin{aligned}
    U_{k+1}^{n+1} - u^{(n+1)J} &= (I+\Delta TA)^{-1}\Big[(U_{k+1}^{n} - u^{nJ})
- (U_{k}^{n} - u^{nJ})\Big] \\
& \quad + F(T_n+(J-1)\Delta t, \Delta t, \widetilde U^{n,J-1}, f)
- F(T_n+(J-1)\Delta t, \Delta t, u^{nJ-1}, f)) \\
&= (I+\Delta TA)^{-1}\Big[(U_{k+1}^{n} - u^{nJ})
- (U_{k}^{n} - u^{nJ})\Big]
+ r(-\Delta t A) (\widetilde U^{n,J-1}- u^{(n+1)J-1})\\
&=\cdots\\
&= (I+\Delta TA)^{-1}\Big[(U_{k+1}^{n} - u^{nJ})
- (U_{k}^{n} - u^{nJ})\Big]
+ r(-\Delta t A)^J (\widetilde U^{n,0}- u^{nJ})\\
&= (I+\Delta TA)^{-1}\Big[(U_{k+1}^{n} - u^{nJ})
- (U_{k}^{n} - u^{nJ})\Big]
+ r(-\Delta t A)^J ( U_k^{n}- u^{nJ}),
\end{aligned}
\end{equation*}
For the sake of simplicity, we define $E_{k}^{n} = U_k^n - u^{nJ}$ and rewrite the above equation as
$$E_{k+1}^{n+1} = (I+\Delta TA)^{-1} (E_{k+1}^n - E_{k}^{n})
+  r(-\Delta t A)^J E_{k}^{n}
$$

Recall that the operator $A$ is a positive definite, selfadjoint, linear operator with a compact inverse,
defined in Hilbert space $(H, (\cdot , \cdot ))$. Then by the spectral theory,
$A$ has positive eigenvalues $\{\lambda_j\}_{j=1}^\infty$, where $0<\lambda_1\le \lambda_2 \le \ldots $ and $\lambda_j\rightarrow\infty$,
and the corresponding eigenfunctions $\{\phi_j\}_{j=1}^\infty$ form an orthonormal basis of the Hilbert space $H$.
Then, letting $e_{k,j}^n = (E_{k}^{n}, \phi_j)$, by means of spectrum decomposition,
we derive
\begin{equation*}
\begin{aligned}
    e_{k+1,j}^{n+1} = \frac{e_{k+1,j}^n-e_{k,j}^n}{1+\Delta T \lambda_j} +  r(-\Delta t \lambda_j)^J e_{k,j}^{n}.
\end{aligned}
\end{equation*}
By letting $d_j = \Delta T \lambda_j$, we have
\begin{equation*}
\begin{aligned}
    e_{k+1,j}^{n+1} \le (1+d_j)^{-1} e_{k+1,j}^n +  (r(-d_j/J)^J - (1+d_j)^{-1})e_{k,j}^{n}.
\end{aligned}
\end{equation*}
We apply the recursion and use the fact that $ e_{k+1,j}^0=0$, and hence obtain 
\begin{equation*}
\begin{aligned}e_{k+1,j}^{n+1}
=&(r(-d_j/J)^J - (1+d_j)^{-1}) e_{k,j}^n
    + (1+d_j)^{-1} e_{k+1,j}^n\\
=&(r(-d_j/J)^J - (1+d_j)^{-1}) \Big(e_{k,j}^n
    + (1+d_j)^{-1} e_{k,j}^{n-1} \Big)
    +  (1+d_j)^{-2}e_{k+1,j}^{n-1}\\
=&\dots \\
=&(r(-d_j/J)^J - (1+d_j)^{-1})
\Big( e_{k,j}^n + (1+d_j)^{-1} e_{k,j}^{n-1} + \dots +
(1+d_j)^{-(n-1)}e_{k,j}^{1}\Big).
 \end{aligned}
\end{equation*}
Now taking the absolute value on the both sides yields
 \begin{equation}\label{eqn:factor-0}
 \begin{aligned}|e_{k+1,j}^{n+1} | \leq &|r(-d_j/J)^J - (1+d_j)^{-1}| \cdot
(1 + (1+d_j)^{-1} + \dots + (1+d_j)^{-(n-1)})
\max_{1\le n\le N}{|e_{k,j}^{n}|} \\
\leq & \frac{|r(-d_j/J)^J -  (1+d_j)^{-1}|}{1-(1+d_j)^{-1}}\max_{1\le n\le N}{|e_{k,j}^{n}|}  \\
=& \left|\frac{(1+d_j) r(-d_j/J)^J-1}{d_j}\right|\max_{1\le n\le N}{|e_{k,j}^{n}|} \\
\le & \sup_{s\in(0,\infty)} \left|\frac{(1+s) r(-s/J)^J-1}{s}\right|\max_{1\le n\le N}{|e_{k,j}^{n}|}.
\end{aligned}
\end{equation}
If the the leading factor is strictly smaller than one, i.e.,
 \begin{equation}\label{eqn:conv-factor}
   \sup_{s\in(0,\infty)} \left|\frac{(1+s) r(-s/J)^J-1}{s}\right| \le \gamma < 1,
\end{equation}
then $e_{k}^{n}$ converges to zero linearly with a factor (smaller than) $\gamma$, and hence
the parareal iteration converges linearly to the time stepping solution \eqref{eqn:semi}
in sense of \eqref{eqn:conv-1}.



Our convergence analysis in this and next sections heavily depend on the constant $\kappa_\alpha$ defined by
\begin{equation}\label{eqn:kappa-alpha}
\kappa_\alpha := \sup_{s\in(0,\infty)}\left|\frac{(1+s)e^{-\alpha s}-1 }{s}\right|,\qquad \text{for} ~~\alpha\in[0,2].
\end{equation}
To begin with, we establish an simple upper bound for the constant $\kappa_\alpha$.

\begin{lemma}\label{lem:exp}
Let $\alpha\in[0,2]$, and $\kappa_\alpha$ be the constant defined in \eqref{eqn:kappa-alpha}.
Then there holds
\begin{equation*}
\begin{split}
\kappa_\alpha \le  \begin{cases}
  e^{\alpha-2} ,\quad &\alpha\in[1,2];\\
  \max(e^{\alpha-2}, 1-\alpha),&\alpha\in[0,1).
\end{cases}
\end{split}
\end{equation*}
\end{lemma}

\begin{proof}
First of all, we show the claim that
\begin{equation}\label{eqn:exp-01}
\frac{(1+s)e^{-\alpha s}-1 }{s}\geq -e^{\alpha-2}.
\end{equation} To this end, we define the auxiliary function
$$g(s) = (1+s)e^{-\alpha s} + e^{\alpha-2}s.$$
Then a simple computation yields
$$g'(s) = (1-\alpha-\alpha s)e^{-\alpha s} + e^{\alpha-2} \quad \text{and}\quad g''(s) =(\alpha^2 + \alpha^2 s - 2\alpha )e^{-\alpha s}.$$
It is easy to observe that $g''(s)$ admits a single root at $s=(2-\alpha)/\alpha$,
and $$g'(x) \geq g'((2-\alpha)/\alpha) = 0.$$
Therefore $g(s)$ is increasing in $[0,\infty)$.
As a result, $g(s) \geq g(0) = 1$, and hence
$${(1+s)e^{-\alpha s}-1 }\geq -e^{\alpha-2}{s}\qquad \forall ~~s\geq0,$$
which implies \eqref{eqn:exp-01}.
Moreover, for $\alpha \geq 1$, we observe that
$$1+s \leq e^s \leq e^{\alpha s}\qquad \forall s\ge0,$$
which immediately leads to $(1+s)e^{-\alpha s}-1 \leq 0$ for all $s\ge0$.
This completes the proof for the desired results in case that $\alpha\in[1,2]$.

Now we turn to the case that
 $\alpha \in[0,1)$. Let $\kappa_\alpha^* = \max\{e^{\alpha-2}, 1-\alpha\}$
 and define $$g(x) = (1+s) e^{-\alpha s}- \kappa_\alpha^*s.$$
Then the  simple  computation yields
$$g'(s) = (1-\alpha - \alpha s) e^{-\alpha s}- \kappa_\alpha^*  \quad \text{and}\quad
 g''(s) = (\alpha^2 + \alpha^2 s - 2\alpha) e^{-\alpha s}.$$
Noting that
$g'(0) = 1-\alpha- \kappa_\alpha^*\le 0$, $g'(\infty) = - \kappa_\alpha^*<0$ and $g'((2-\alpha)/\alpha) \le 0$.
These imply  $g'(s) \leq 0$ and hence $g$ is decreasing function in $[0,\infty)$.
Therefore $g(s) \leq g(0) = 1$, which further implies
$$(1+s) e^{-\alpha s}- \kappa_\alpha^* s \leq 1. $$
This leads to the desired assertion for the case that  $\alpha \in[0,1)$.
\end{proof}

Lemma \ref{lem:exp} only provides a rough upper bound for $\kappa_\alpha$. In fact, for a fixed $\alpha$
we can further improve the upper bound via a more careful computation. In Figure \ref{fig:ka-alpha}, we numerically compute the constant $\kappa_\alpha$
for $\alpha\in[0,2]$ and plot those values.

\begin{figure} [tbhp]
\centering
\includegraphics[width=0.6\textwidth]{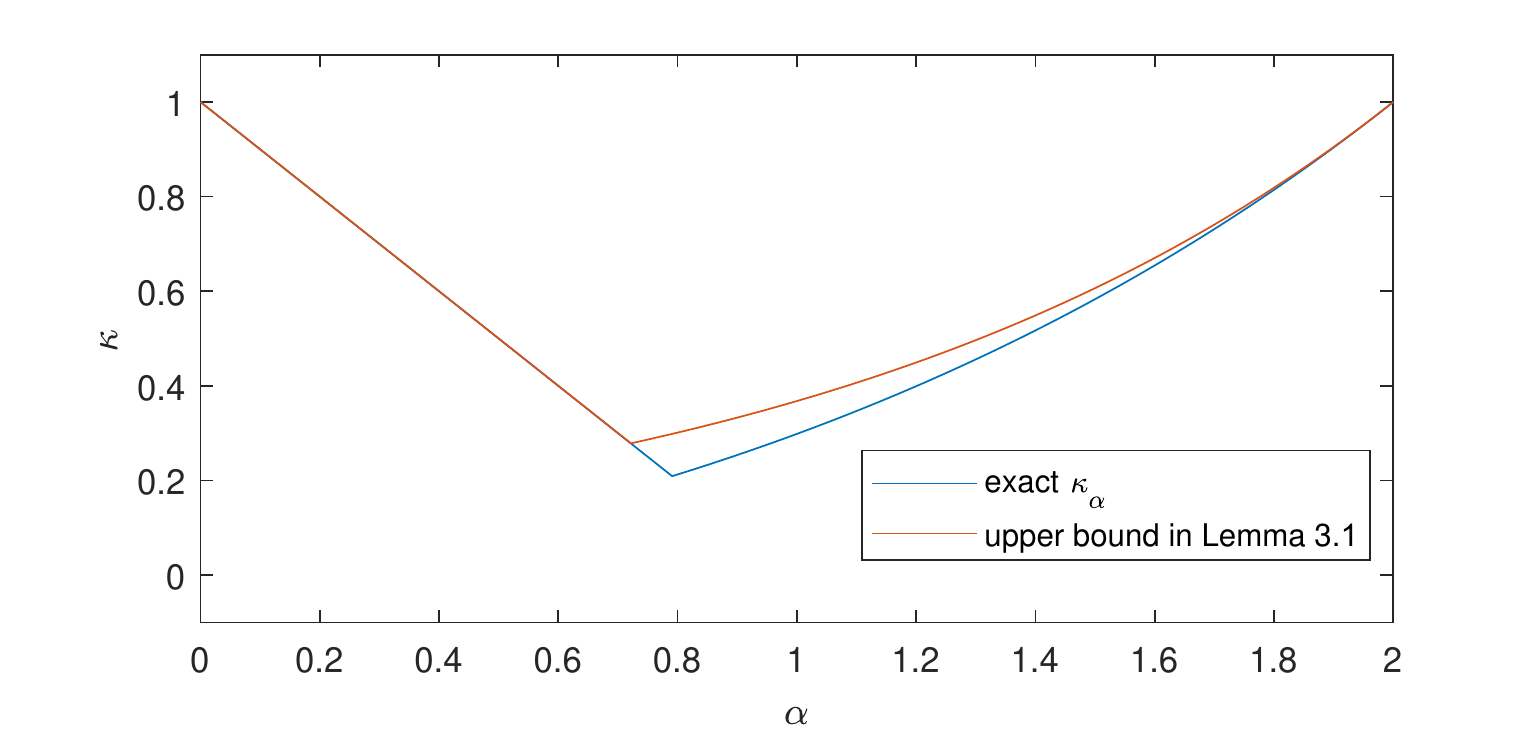}
   \caption{Plot of $\kappa_\alpha$ defined in lemma \ref{lem:exp}.}
   \label{fig:ka-alpha}
\end{figure}

The next lemma provides a sharper estimate for $\kappa_1$.

\begin{lemma}\label{lem:exp-1}
Let $\kappa_\alpha$ be the constant defined in \eqref{eqn:kappa-alpha}. Then $\kappa_1\approx 0.2984$.
\end{lemma}

\begin{proof}
To show the sharp estimate, we note that for $g(s) = 1-(1+s)e^{- s}$, there holds
$g'(s) = se^{-s} \ge 0$ and $g(0) = 1$. Therefore $g(s)\ge0$ for all $s\in(0,\infty)$.
This further implies $\psi(s)=(1-(1+s)e^{- s})/s \ge 0$ for $s\in (0,\infty)$.

Meanwhile, we note that
$$\varphi(s) := s^2 e^s \psi'(s) =  s^2 +s+1 -e^s. $$
It is easy to verify that $\varphi(s)$ has a unique root, denoted by $s_*$, in $(0,\infty)$, so does $\psi'(s)$. Therefore, $\kappa_1 = f(s_*)$.
Using the Newton's algorithm, we find $s_*\approx 1.793$ and hence $\kappa_1\approx 0.2984$.
\end{proof}

Now we state our main theorem which verifies the desired result \eqref{eqn:conv-factor} with $\gamma\approx 0.3$.

\begin{theorem}\label{thm:conv-1}
Let conditions (P1)-(P3) hold valid. Then there exists a threshold $J_*>0$ such that
for all $J\ge J_*$
 \begin{equation*}
\sup_{s\in(0,\infty)}  \left|\frac{(1+s) r(-s/J)^J-1}{s}\right| \le 0.3,
\end{equation*}
        \end{theorem}
\begin{proof}
First of all, we aim to show that
 \begin{equation}\label{eqn:lim}
  \lim_{J\rightarrow\infty}\sup_{s\geq0}\frac{(1+s)}{s} \left|r(-s/J)^J-e^{- s}\right|=0.
  \end{equation}
For a given $\delta>0$, for any $s>\delta$,  it is obvious that $(1+s)/s$
is bounded by a constant $C(\delta)$. Meanwhile, note that
conditions (P1)-(P3) are fullfilled.
Then by means of the nonsmooth data error estimate \cite[Theorem 7.2]{Thomee:2006}, there holds
\begin{equation*}
\left|r(-s/J)^J-e^{-  s}\right| \le c J^{-q},
\end{equation*}
where $c$ is independent of $s$.
Then we derive
\begin{equation*}
\lim_{J\rightarrow\infty}\sup_{s> \delta}\frac{(1+s)}{s} \left|r(-s/J)^J-e^{-  s}\right|=0.
\end{equation*}
For $0 < s \leq\delta$, conditions (P1) and (P2) imply
\begin{equation*}
\begin{aligned}
   &\quad \frac{(1+s)}{s}   \left|r(-s/J)^J-e^{-  s}\right| \\
    &= \dfrac{(1+s)}{s}|r(-s/J)^{-1}-e^{- s/J}|
    \left|\sum_{i=0}^{J-1} r(-s/J)^{-i} e^{ -(J-1-i)  s/J}\right| \\
    &\leq \dfrac{(1+s)}{s} \cdot C(\alpha, \delta)\Big(\frac{s}{J}\Big)^{q+1} \cdot J \\
    &\leq C(\alpha, \delta) J^{-q}.
\end{aligned}
\end{equation*}
Therefore we arrive at
$$\lim_{J\to\infty}\sup_{0<s\leq\delta}\frac{(1+s)}{s} \left|r(-s/J)^{-J} -e^{- s}\right|=0,$$
which completes the proof of the \eqref{eqn:lim}.
This together with Lemma \ref{lem:exp-1} implies that
    $$\lim_{J\to\infty}\sup_{s\in(0,\infty)}
    \left| \frac{(1+s)r(-s/J)^J-1}{s}  \right| = \kappa_1 < 0.3,$$
which completes the proof of the lemma.
\end{proof}

Next, using Theorem \ref{thm:conv-1}, we are able to show the linear convergence of the parareal iteration \ref{alg:para}.

\begin{theorem}\label{thm:conv-2}
Let conditions (P1)-(P3) be fullfilled and the data regularity in Lemma \ref{lem:conv-00} hold valid. Let $u^n$ be the solution to the time stepping scheme \eqref{eqn:semi},
and $U_k^n$ be the solution obtained from the parareal algorithm \ref{alg:para}.  Then there exists a threshold $J_*>0$ such that
for all $J\ge J_*$,
 we have
$$ \max_{1 \le n\le N_c} \| U_k^n - u^{nJ} \|  \le c \gamma^k\quad \text{with} \quad \gamma=0.3.$$
\end{theorem}

\begin{proof}
In Algorithm \ref{alg:para}, the initial guess $U_0^n$ is obtained by the coarse propagator, i.e., the backward Euler scheme.
Then Lemma \ref{lem:conv-00} (with $q=1$) implies the estimate
\begin{equation}\label{eqn:conv-3}
\begin{aligned}
 \| U_0^n - u^{nJ} \| &\le \| U_0^n - u(T_n)\|  + \|u(T_n)- u^{nJ} \| \\
 &\le c \big((\Delta T) T_n^{-1} + (\Delta t) t_{nJ}^{-1} \big)\le c n^{-1}.
\end{aligned}
\end{equation}
Let $E_k^n =U_k^n - u^{nJ}$ and $e_{k,j}^n = (E_k^n, \phi_j)$. The the relation \eqref{eqn:factor-0} and
Theorem \ref{thm:conv-1} imply
\begin{equation*}
\begin{aligned}
\sup_{1\le n\le N_c} \|  E_k^n \|^2
& \le  \sum_{j=1}^\infty \sup_{1\le n\le N_c} | e_{k,j}^n |^2 \le \gamma^2  \sum_{j=1}^\infty \sup_{1\le n\le N_c} | e_{k-1,j}^n |^2\\
& \le \cdots     \le \gamma^{2k}  \sum_{j=1}^\infty \sup_{1\le n\le N_c} | e_{0,j}^n |^2
    \end{aligned}
\end{equation*}
with $\gamma = 0.3$. This together with the estimate
$
  \sup_{1\le n\le N_c} | e_{0,j}^n |^2 
   \le  \sum_{n=1}^{N_c} | e_{0,j}^n |^2
$
leads to
\begin{equation*}
\begin{aligned}
\sup_{1\le n\le N_c} \|  E_k^n \|^2  &\le c \gamma^{2k}  \sum_{j=1}^\infty  \sum_{n=1}^{N_c} | e_{0,j}^n |^2
 \le  c \gamma^{2k}   \sum_{n=1}^{N_c} \| E_{0}^n \|^2
  \le  c \gamma^{2k}   \sum_{n=1}^{N_c} n^{-2}
 \le c \gamma^{2k} ,
    \end{aligned}
\end{equation*}
where in the second last inequality we apply the estimate \eqref{eqn:conv-3}.
 This completes the proof of the theorem.
\end{proof}

\begin{remark}\label{rem:conv-1}
Theorem \ref{thm:conv-2} provides an useful upper bound of the convergence factor for all single step integrators
(satisfying (P1)-(P3)),
which might not be sharp for specific one. For example, in \cite[Lemma 4.3]{MathewSarkisSchaerer:2010},
Mathew, Sarkis and Schaerer considered the backward Euler method and proved
that the convergence factor of the Parareal algorithm is around $0.298$ (with $J_*=2$).
In \cite{Wu:2015}, Wu showed that convergence factors are $0.316$ (with $J_*=2$) and $0.333$ (with $J_*=2$)
for the second-order diagonal implicit Runge--Kutta method
and a single step TR/BDF2 method (i.e., the ode23tb solver for ODEs in MATLAB), respectively.
These error bounds might be slightly improved by increasing $J_*$.
See also \cite{WuZhou:2015} for the analysis for a third-order diagonal implicit Runge--Kutta method with a convergence factor $0.333$ and $J_*=4$.
\end{remark}

\begin{remark}\label{rem:conv-2}
Theorem \ref{thm:conv-2} only provides the existence of the threshold $J_*$ without any upper bound estimate.
It is obvious that a huge $J_*$ may destroy the parallelism of the algorithm.
Then a question arise naturally: is it possible to find $J_*$ for a given scheme satisfying conditions (P1)-(P3)? This is the focus of Section \ref{sec:RK}.
\end{remark}

\section{Case studies for several high-order single step integrators}\label{sec:RK}

In this section, we shall study some popular single step methods. As we mentioned in Remark \ref{rem:conv-2},
Theorem \ref{thm:conv-2} did not provide an sharp estimate for the threshold $J_*$.
In fact, there is no universal estimate for all single step methods. Fortunately,
for any given single step integrator satisfying conditions (P1)-(P3) and fixed convergence rate
$\gamma>0.2984$,
we have a regular routine to find a sharper estimate for $J_*$.

We consider three time-stepping methods,
namely the the two-, three-, four-stage Lobatto IIIC methods, which are respectively second-, fourth- and sixth-order accurate, to the
initial and boundary value problem \eqref{eqn:pde}.
For the reader's convenience, we present the Butcher tableaus of the
two-, three-, four-stage Lobatto IIIC methods, respectively,
%
\begin{equation}
\label{2-stage}
\begin{tabular}{cc|c}
 $\frac12$ &  $-\frac12$  & $0$\\[2pt]
 $\frac12$  & $\frac12 $& $1$ \\[2pt]
\hline
 $\frac12 $  & $\frac12$  & $\; $ \\
\end{tabular}
\ =:
\begin{tabular}{c|c}
 $\AA$ & $c$ \\[1pt]
\hline
 $b^\top$  & $\vphantom{\sum^{\sum^\sum}} $ \\
\end{tabular}
\end{equation},

\begin{equation}
\label{3-stage}
\begin{tabular}{ccc|c}
 $\frac16$ &  $-\frac31$
 & $\frac16$ & $0$ \\[2pt]
 $\frac16$  & $\frac5{12}$
 & $-\frac1{12}$ & $\frac12$ \\[2pt]
 $\frac16$  & $\frac23$
 & $\frac16$ & $1$ \\*[3pt]
\hline
 $\frac16 $  & $\frac23$ & $ \frac16$   & $\; $ \\
\end{tabular}
\ =:
\begin{tabular}{c|c}
 $\AA$ & $c$ \\[1pt]
\hline
 $b^\top$  & $\vphantom{\sum^{\sum^\sum}} $ \\
\end{tabular}
\end{equation}
and
\begin{equation}
\label{4-stage}
\begin{tabular}{cccc|c}
 $\frac1{12}$ &  $-\frac {\sqrt{5}}{12}$
 & $\frac {\sqrt{5}}{12}$ & $-\frac1{12}$ & $0$ \\*[2pt]
 $\frac1{12}$  & $\frac14 $ & $\frac{10-7\sqrt{5}}{60}$
 & $\frac{\sqrt{5}}{60}$
 & $\frac12-\frac{\sqrt{5}}{10}$\\*[2pt]
 $\frac1{12}$  & $\frac{10+7\sqrt{5}}{60}$
 &$\frac14$& $-\frac{\sqrt{5}}{60}$
 &  $\frac12+\frac{\sqrt{5}}{10}$ \\*[2pt]
  $\frac1{12}$ & $\frac5{12}$ &  $\frac5{12}$
 & $\frac1{12}$ & $1$ \\[3pt]
\hline
   $\frac1{12}$ & $\frac5{12}$ &  $\frac5{12}$
 & $\frac1{12}$ &  \\
\end{tabular}
\ =:
\begin{tabular}{c|c}
 $\AA$ & $c$ \\[1pt]
\hline
 $b^\top$  & $\vphantom{\sum^{\sum^\sum}} $ \\
\end{tabular}
\,.
\end{equation}

Let us also briefly recall some  well-known facts about Lobatto IIIC; for details we refer to \cite{HW}.
These methods can be viewed as discontinuous collocation methods.
The \textsl{order} of the $m$-stage Lobatto IIIC methods is
$q=2m-2$.
In particular, the methods are \textsl{algebraically stable} and
{\textsl{L-stable}}, that makes them suitable for stiff problems.
The stability functions $r$,
\[r(z):=1+z b^\top (I-z\AA)^{-1} \mathbbm{1}\quad\text{with}\quad  \mathbbm{1}:=(1,\dotsc,1)^\top\in \R^q,\]
is given by the $(m-2,m)$-Pad\'e
approximation to $e^z$ and
vanishes at infinity, i.e., $r(\infty)=1-b^T \AA^{-1} \mathbbm{1}=0.$
Note that the computational cost of implicit Runge-Kutta methods increases fast with the stage number, and we refer
to \cite{Butcher:1976, RK:1988, JN:1995} and the reference therein for some efficient implementations.

The following argument highly depends on the upper bound for the constant $\kappa_\alpha$
defined in \eqref{eqn:kappa-alpha}. From Figure \ref{fig:ka-alpha},
we observe that Lemma \ref{lem:exp} gives an sharp estimate
for $\kappa_\alpha$ for $\alpha<0.7$, while the estimate for $\alpha>1$ could be further improved. The next lemma provides an estimate for $\alpha=1.02$, which is useful in the analysis of convergence rate.

\begin{lemma}\label{lem:exp-102}
Let $\kappa_\alpha$ be the constant defined in \eqref{eqn:kappa-alpha}. Then $\kappa_{1.02} \approx 0.3078<0.31$.
\end{lemma}

\begin{proof}
With $\beta=1.02 \geq 1$ and $\psi(s) = \frac{(1+s)e^{-\beta s} - 1}{s}$,
we observe that
$\psi(0+) = 1-\beta$ and $\psi(\infty)= 0$.
Meanwhile, since $e^{-\beta s} \leq e^{-s} \leq (1+s)^{-1}$ for $s\ge0$,
we derive that $\psi(s) \leq 0$.
Now we intend to show that $\psi'(s)$ admits a unique root in $(0,\infty)$, denoted as $x_*$. Then $\kappa_{\beta} = \psi(x_*)$.
Noting that
$$\psi'(s) =
\frac{1-(1 + \beta s + \beta s^2)e^{-\beta s}}{s^2},$$
it suffices to show that
$g(s) = 1-(1 + \beta s + \beta s^2)e^{-\beta s}$ has a unique root in $(0,\infty)$.
It is straightforward to see that the function
$$g'(s) = (-2+\beta + \beta s)\beta s e^{-\beta s}$$
admits a unique root in $(0,\infty)$. Then by the fact that $g(0)=0$ and Rolle's theorem, we conclude that $g$ has at most one root in $(0,\infty)$.
Meanwhile, we observe $\psi'(1)=1-(1+2\beta)e^{-\beta} \approx -0.0962 <0$ and
$\psi'(2) = \frac{1-(1+6\beta)e^{-2\beta}}{4} \approx 0.01855 >0$.
Therefore, there exists a unique root of $\psi'$ in $(0, \infty)$, named as $x_*$,
which lies in $(1,2)$. Using the Newton’s algorithm,  we find $x_*\approx 1.715$ and hence $\kappa_\beta = f(x_*) \approx 0.3078 \le 0.31$.
\end{proof}

\begin{proposition}\label{prop:lobatto-2}
Let $u^n$ be the solution to the time stepping scheme \eqref{eqn:semi} using
the two-stage Lobatto IIIC method \eqref{2-stage},
and $U_k^n$ be the solution obtained from the parareal algorithm \ref{alg:para}. Then for all $J\ge 2$, there holds
$$ \max_{1 \le n\le N_c} \| U_k^n - u^n \|  \le c \gamma^k\quad \text{with} \quad \gamma=0.31.$$
\end{proposition}

\begin{proof}
It suffices to show that for any $J\ge 2$, there holds
\begin{equation}\label{eqn:lobatto-2}
    \sup_{s\in(0,\infty)}  \left|\frac{(1+s) r(-s/J)^J-1}{s}\right| \le 0.31,
    \quad
    \text{where}~~ r(-s) = \frac{2}{s^2+2s+2}.
\end{equation}
To this end, we define $\alpha = 0.69$ and $\beta=1.02$.
Then Lemma \ref{lem:exp} implies that
$\kappa_\alpha \le 1-0.69 = 0.31$,
and meanwhile Lemma \ref{lem:exp-102} indicates $\kappa_\beta \le 0.31$.

Next, we aim to show that $e^{-\beta s} \le r(-s) \le e^{-\alpha s}$ for all $s\in(0,s_*)$ with $s_*=3.2$. First of all, using the fact that
$\frac{2}{s^2+2s+2} \ge {e^{-s}} \ge e^{-\beta s}$ for all $s > 0$, we derive
the first inequality $e^{-\beta s} \le r(-s)$.
Then we turn to the second inequality
$r(-s) \le e^{-\alpha s}$, equivalent to
$g(s) := 2e^{\alpha s} -(s^2+2s+2) \le 0$ in  $(0, s_*)$.
Noting that $g''(s) = 2\alpha^2 e^{\alpha s} - 2$, which admits a
unique root at $-\frac{2\ln \alpha}{\alpha}$.
Meanwhile, we observe that
$g''(s) <0$ in $(0, -\frac{2\ln \alpha}{\alpha})$,
and $g''(s) >0$ in $( -\frac{2\ln \alpha}{\alpha}, \infty)$.
Besides, since $g'\left( -\frac{2\ln \alpha}{\alpha}\right)  =
\frac{4\ln \alpha +2-2\alpha}{\alpha}<0$
and $g'(0)=2\alpha-2 <0$, we conclude that
$g'(s) < 0$ in $(0, -\frac{2\ln \alpha}{\alpha})$,
and $g'(s)$ has a unique root in $(-\frac{2\ln \alpha}{\alpha}, \infty)$. Moreover, the facts $g'(2) \approx -0.5 <0$ and $g'(3) \approx 2.9 >0$ implies that
there exists a constant $s_1 \in (2,3)$, s.t. $g'(s_1)= 0$,
and $g'(s) <0$ in $(0, s_1)$,
$g'(s) >0$ in $(s_1, \infty)$.
Then we note that $g(s_*) = -0.445 <0$ and conclude that
$g(s) := 2e^{\alpha s} -(s^2+2s+2) \le 0$ in  $(0, s_*)$,
which implies $r(s) \le e^{-\alpha s}$ in $(0, s_*)$.
As a result, we arrive at $e^{-\beta s} \le r(-s/J)^J \le e^{-\alpha s}$
for all $s\in(0, J s_*)$ which  implies
\begin{equation*}
    \sup_{s\in(0, J s_*)}
    \left|\frac{(1+s) r(-s/J)^J-1}{s}\right| \le 0.31.
\end{equation*}

Besides, we observe the fact that
$$\sup\limits_{(s_*, \infty)}
\Big|\frac{2}{s^2 + 2s + 2}\Big|
= \Big|\frac{2}{s_*^2 + 2s_* + 2}\Big|
\approx 0.1073 < 0.11.$$
Then we derive for $s\in (Js_*,\infty)$ and $J\ge 2$
\begin{equation*}
\begin{split}
    \left|\frac{(1+s) r(-s/J)^J-1}{s}\right|
    &\le \frac{1+s}{s}|r(-s/J)^J|+s^{-1}
    \le \frac{1+Js_*}{Js_*} (0.11)^J + (Js_*)^{-1}    \\
    &\le  \frac{1+6.4}{6.4} (0.11)^2 + (2\times 3.2)^{-1}
    \approx 0.1702 \le 0.31.
\end{split}
\end{equation*}
This completes the proof of \eqref{eqn:lobatto-2}.
\end{proof}

\begin{proposition}\label{prop:lobatto-3}
Let $u^n$ be the solution to the time stepping scheme
\eqref{eqn:semi} using
the three-stage Lobatto IIIC method \eqref{3-stage},
and $U_k^n$ be the solution obtained from the parareal algorithm \ref{alg:para}. Then for all $J \ge 2$, there holds
$$ \max_{1 \le n\le N_c} \| U_k^n - u^n \|  \le c \gamma^k\quad \text{with} \quad \gamma=0.31.$$
\end{proposition}

\begin{proof}
Similar to the proof of proposition \ref{prop:lobatto-2}, we aim to show that for any $J\ge 6$
\begin{equation}\label{eqn:lobatto-3}
    \sup_{s\in(0,\infty)}  \left|\frac{(1+s) r(-s/J)^J-1}{s}\right| \le 0.31,
    \quad
    \text{where}~~ r(-s) = \frac{24-6s}{s^3 + 6s^2 + 18s + 24}.
\end{equation}
Letting $\alpha = 0.69$ and $\beta=1.02$,
Lemmas \ref{lem:exp} and \ref{lem:exp-102} implies that
$\kappa_\alpha \le 1-0.69 = 0.31$ and $\kappa_\beta \le 0.31$, respectively.

Next, we show the claim that $e^{-\beta s} \le r(-s) \le e^{-\alpha s}$ for all $s\in(0,s_*)$ with $s_*=2$. To begin with, we shall prove that $r(-s) > e^{-\beta s}$ for all $s\in(0,\infty)$, which is equivalent to
$$\psi(s) = e^{\beta s} (-6s+24) - (s^3+6s^2+18s+24) >0
\quad \forall ~s\in (0, s_*).$$
We note that
$$\psi^{(4)}(s) = 6\beta^3e^{\beta s}(-\beta s + 4\beta - 4).$$
has a unique root in $(0,\infty)$, namely
$s_0= \frac{4\beta - 4}{\beta} \approx0.0784$, and hence
$\psi^{(4)}(s)> 0$ for all $s\in(0,s_0)$.
Besides, we observe that
$\psi^{(3)}(0) = 0.742 >0$, $\psi^{(3)}(s_0) \approx 0.762 >0$.
Therefore $\psi^{(3)}(s)$ has a unique root
in $(s_0, +\infty)$,
denoted as $s_1$. By means of Newton's algorithm, we know that
$s_1\approx 0.4832$.
Similarly, since $\psi''(0) =  0.730>0$ and $\psi''(s_1) =  1.00>0$, we conclude that $\psi''(s)$ is always positive in $[0,s_1]$ and it
has a unique root $s_2\in (s_1,\infty)$, and we find $s_2\approx 0.9980$.
Repeating the argument, we are able to show that
$\psi'(s)$ keeps positive in $[0,s_2]$ and
the unique root in $(s_2,\infty)$ locates at
$s_3 \approx 1.5344$.
Finally, we observe that $\psi(0) = 0$ and $\psi(s_3) \approx 1.401>0$, so $\psi$ is positive in $(0,s_3]$
and it admits a unique root at $s_4\in(s_3,\infty)$.
Noting that $\psi(s_*) \approx 0.2873 > 0$,
we conclude that $r(-s) > e^{-\beta s}$ for all $s\in(0,s_0)$.

Next we will show that $r(-s) < e^{-\alpha s}$ in $(0, s_*)$, which is equivalent to show
\begin{equation*}
    \varphi(s) = e^{\alpha s} (-6s+24) - (s^3+6s^2+18s+24) < 0 \quad \text{for all} ~s\in (0, s_*).
\end{equation*}
We note the fact that
\begin{equation*}
\varphi^{(4)}(s) = 6\alpha^3e^{\alpha s}(-\alpha s + 4\alpha - 4)<0\quad \text{for all} ~s\in (0, \infty).
\end{equation*}
Meanwhile, we have $\varphi^{(4)}(0)<0$, $\varphi^{(3)}(0)<0$, $\varphi''(0)<0$, $\varphi'(0)<0$, and $\varphi(0)<0$.
Those together imply $\varphi(s) <0$ for any
$s\in (0, \infty)$. Therefore $r(-s) < e^{-\alpha s}$.

As a result, for any $J\ge 2$, we arrive at $e^{-\beta s} \le r(-s/J)^J \le e^{-\alpha s}$
for all $s\in(0, J s_*)$, that further implies the estimate
\begin{equation*}
    \sup_{s\in(0, J s_*)}
    \left|\frac{(1+s) r(-s/J)^J-1}{s}\right| \le 0.31.
\end{equation*}

Next, we aim to prove the claim that
\begin{equation}\label{claim:bound-3}
   \sup\limits_{(s_*, \infty)}
   \big|r(-s)\big| = \sup\limits_{(s_*, \infty)}
   \Big|\frac{24-6s}{s^3 + 6s^2 + 18s + 24}\Big|
   \le  0.15.
\end{equation}
To begin with, we show that $\frac{\d}{\d s}r(-s)$ admits a unique root in $(2, \infty)$. 
We note that
$$\frac{\d}{\d s}r(-s) = \frac{12(s^3-3s^2-24s-48)}
{(s^3 + 6s^2 + 18s + 24)^2}.$$
and hence it is sufficient to show that
$\eta(s) = s^3-3s^2-24s-48$
has a unique root in $(2, \infty)$.
Since $\eta'(s)$ has two roots, $-2$ and $4$,
and $\eta(-2)=-20<0$, $\eta(4)=-128<0$,
we conclude that $\eta(s) <0$ for all $s\in[-2,4]$, and
$\eta(s)$ admits a unique root in $(4, \infty)$,
namely $s_5 \approx 7.235$. 
Therefore $r(-s)$ is decreasing in $(s_*,s_5)$
and increasing in $[s_5,\infty)$.
Noting that fact that $r(-s_*) \approx 0.130$,
$r(-s_5)\approx -0.0229$ and $r(-\infty)=0$, we obtian
$
    \sup\limits_{(s_*, \infty)}
    \big|r(-s)\big| = \big|r(-s_*)\big|
    \le  0.15.
$

Therefore, we derive for $s\in (Js_*,\infty)$ and $J\ge 2$
\begin{equation*}
\begin{split}
    \left|\frac{(1+s) r(-s/J)^J-1}{s}\right|
    &\le \frac{1+s}{s}|r(-s/J)^J|+s^{-1}
    \le \frac{1+Js_*}{Js_*} (0.15)^J + (Js_*)^{-1}    \\
    &\le  \frac{5}{4} (0.02)^2 + 4^{-1}
    \approx 0.251 \le 0.31.
\end{split}
\end{equation*}
This completes the proof of \eqref{eqn:lobatto-3} as well as the proposition.
\end{proof}

\begin{proposition}\label{prop:lobatto-4}
Let $u^n$ be the solution to the time stepping scheme
\eqref{eqn:semi} using
the four-stage Lobatto IIIC method \eqref{4-stage},
and $U_k^n$ be the solution obtained from the parareal algorithm \ref{alg:para}. Then for all $J \ge 2$, there holds
$$ \max_{1 \le n\le N_c} \| U_k^n - u^n \|  \le c \gamma^k\quad \text{with} \quad \gamma=0.31.$$
\end{proposition}
\begin{proof}
Similar to the proof of Proposition \ref{prop:lobatto-2}, we aim to show that for any $J\ge 2$
\begin{equation}\label{eqn:lobatto-4}
    \sup_{s\in(0,\infty)}  \left|\frac{(1+s) r(-s/J)^J-1}{s}\right| \le 0.31,
    ~~
    \text{where}~~ r(-s) =
    \frac{12s^2-120s+360}{s^4+12s^3+72s^2 + 240s + 360}.
\end{equation}
Letting $\alpha = 0.69$ and $\beta=1.02$,
Lemmas \ref{lem:exp} and \ref{lem:exp-102} implies that
$\kappa_\alpha \le 1-0.69 = 0.31$ and $\kappa_\beta \le 0.31$, respectively. Next, we show the claim that
\begin{equation}\label{claim:sand}
    e^{-\beta s} \le r(-s) \le e^{-\alpha s}~~ \text{for all}~~s\in(0,s_*)~~ \text{with}~s_*=6.8.
\end{equation}

To begin with, we show that, for $s\in(0,\infty)$, $r(-s) \ge e^{-\beta s}$, which is equivalent to
$$\psi(s) = (12s^2-120s+360)e^{\beta s} -(s^4+12s^3+72s^2 + 240s + 360)
 \ge 0.$$
Define $h(s) = \beta^2s^2 + (10\beta-10\beta^2)s+(30\beta^2-50\beta+20)$,
then we have
$$\psi^{(5)}(s) = 12\beta^3e^{\beta s}
\left[\beta^2s^2 + (10\beta-10\beta^2)s+(30\beta^2-50\beta+20)\right]
=12\beta^3e^{\beta s}h(s).$$
Here $h(s)$ is a quadratic polynomial, whose
minimum locates at $\frac{5\beta-5}{\beta}$.
Therefore $h(s) \geq h (\tfrac{5\beta-5}{\beta} ) > 0$.
Then $\psi^{(5)}(s) = 12\beta^2e^{\beta x}h(s)-24 > 12\times 2-24>0$.
Meanwhile, simple computation yields
$$ \psi(0)=0 \qquad \text{and}
\qquad \psi^{(k)}(0)>0~~\quad \text{with}~~ \quad
1\le k\le 5.  $$
Then we conclude that  $\psi(s) > 0$ for all  $s\in (0, \infty)$, and hence
 $r(-s) \ge e^{-\beta s}$ in $(0, \infty)$.

Next we show the bound that $r(s) \le e^{-\alpha s}$ for $s\in (0,s_*)$, which is equivalent to show
$$\varphi(s) = (12s^2-120s+360)e^{\alpha s}
-(s^4+12s^3+72s^2 + 240s + 360) \le 0 \qquad \forall~s\in (0,s_*).$$
Similar to the preceding argument, let $g(s) = (20-50 \alpha + 30 \alpha^2  +10\alpha s -10\alpha^2s +\alpha^2s^2)$. Then
$$  \varphi^{(5)}(s) = 12\alpha^3e^{\alpha s}
(20-50 \alpha + 30 \alpha^2
+10\alpha s -10\alpha^2s +\alpha^2s^2)
= 12\alpha^3e^{\alpha s}  g(s). $$
Here $g$ is a quadratic polynomial with minimum at $\frac{5\alpha-5}{\alpha}$. Therefore, $g(s)\ge g(\tfrac{5\alpha-5}{\alpha}) \approx -2.62$.
Meanwhile, we observe that $g(0) = -0.217 <0$, so
there is a unique root of $g$ in $(0, \infty)$.
It is easy to find that, by means of Newton's algorithm,
that root locates at $s_0\approx0.0993$. Then $\varphi^{(5)}(s) \le 0$
for all $s\in [0,s_0]$ and $\varphi^{(5)}(s) \ge 0$ for $s\in (s_0,\infty)$. Noting that $\varphi^{(4)}(s_0) < 0$ and $\varphi^{(4)}(0) < 0$, so $\varphi^{(4)}(s) < 0$ in $[0,s_0]$ and $\varphi^{(4)}$ admits a unique root in $(s_0,\infty)$, named as $s_1$.
Then Newton's algorithm implies $s_1\approx  1.6849$.
Repeating this argument, we are able to show that
$\varphi^{(3)}(s) < 0$ in $[0,s_1]$ and $\varphi^{(3)}$ has a unique root $(s_1,\infty)$, namely $s_2\approx 3.0558$.
Then we derive that
$\varphi{''}(s) < 0$ in $[0,s_2]$ and  $\varphi{''}(s)$
has a unique root in $(s_2,\infty)$,
denoted as $s_3\approx 4.3640$.
Similarly, $\varphi{'}(s) < 0$ in $[0,s_3]$ and  $\varphi{'}(s)$
has a unique root in $(s_3,\infty)$, named as $s_4\approx 5.6285$.
Finally, since $\psi(0)=0$ and $\varphi(s_4)<0$, we conclude that
$\varphi(s)<0$ in $(0,s_4]$, and $\varphi(s)$ has a unique root in $(s_4,\infty)$. Then the fact that $\varphi(s_*) \approx -447.65<0$ implies $\varphi(s)<0$ in $(0,s_*)$.
This completes the proof of the claim \eqref{claim:sand}.
As a result, for any $J\ge 2$, we arrive at $e^{-\beta s} \le r(-s/J)^J \le e^{-\alpha s}$
for all $s\in(0, J s_*)$ which  implies
\begin{equation*}
    \sup_{s\in(0, J s_*)}
    \left|\frac{(1+s) r(-s/J)^J-1}{s}\right| \le 0.31.
\end{equation*}

Next, we intend to show the claim that
\begin{equation}\label{claim:bound-4}
 \sup\limits_{(s_*, \infty)}
   \big|r(-s)\big| =   \sup\limits_{(s_*, \infty)}  \big|\frac{12s^2-120s+360}{s^4+12s^3+72s^2 + 240s + 360}\Big| \le  0.02.
\end{equation}
In order to establish a bound for the supremum, we note
$$\frac{d}{ds}r(-s)
= \frac{-24s^5+216s^4+1440s^3-1440s^2-43200s-129600}
{(s^4+12s^3+72s^2 + 240s + 360)^2}, $$
and we will show that it admits a unique root in $(s_*, \infty)$,
denoted by $s_5$. Noting that, with
$\mu(s) = -s^5+9s^4+60s^3-60s^2-1800s-5400, $
we have
$$\frac{d}{ds}r(-s) = \frac{24\mu(s)}{(s^4+12s^3+72s^2 + 240s + 360)^2}.$$
So it suffices to show that $\mu(s)$ admits a unique  root in $(s_*, \infty)$. Since
$\mu^{(3)}(s)=-60s^2+216s+360,$ being a quadratic polynomial,
it gains the maximum at $1.8$ where
$\mu^{(3)}(1.8) = 554.4>0$.
Noting that $s_*>1.8$ and $\mu^{(3)}(s_*)=-945.6<0$,
we conclude that $\mu^{(3)}(s)<0$ for all $s\in(s_*, \infty)$.
Moreover, since $\mu''(s_*) \approx 1033.3 >0$ and $\mu''(\infty) =-\infty$, $\mu''(s)$ admits a unique root $s_6\approx 7.6503 \in (s_*, \infty)$. Then we know that $\mu''(s) > 0$ in $(s_*, s_6)$
and $\mu''(s) < 0$ in $(s_6, \infty)$.
Similarly, we have $\mu'(s_*) >0$ and $\mu'(s_6) >0$,
so  $\mu'(s) > 0$ in $[s_*, s_6]$. This together with the fact  $\mu'(\infty) <0$ implies that
$\mu'(s)$ has a unique root $s_7\approx 10.166 \in (s_6, \infty)$.
Finally, using the facts that $\mu(s_*)>0$ and $\mu(s_7)>0$,
we know $\mu(s) >0 $ in  $(s_*, s_7)$ and
 $\mu(s)$ has a unique root $s_8 \approx 12.28 \in (s_7, \infty)$.
Therefore $r(-s)$ is increasing in $(s_*, s_8)$,
and decreasing in $(s_8, \infty)$.
Noting that $r(-s)\approx 0.0088$, $r(-s_8)\approx 0.0118$, and $r(-\infty)=0$, we arrive at
$\sup\limits_{(s_*, \infty)}
\big|r(-s)\big| \le 0.02.$

As a result, the estimate \eqref{claim:bound-4} implies that
for $s\in (Js_*,\infty)$ and $J\ge 2$
\begin{equation*}
\begin{split}
    \left|\frac{(1+s) r(-s/J)^J-1}{s}\right|
    &\le \frac{1+s}{s}|r(-s/J)^J|+s^{-1}
    \le \frac{1+Js_*}{Js_*} (0.02)^J + (Js_*)^{-1}    \\
    &\le  \frac{14.6}{13.6} (0.02)^2 + 13.6^{-1}
    \approx 0.074 \le 0.31.
\end{split}
\end{equation*}
This completes the proof of \eqref{eqn:lobatto-4} as well as the proposition.
\end{proof}

\begin{remark}
Propositions \eqref{prop:lobatto-2}-\eqref{prop:lobatto-4} show that, for two-, three-, four-stage Lobbatto IIIC schemes, the convergence factor is (at worst) $0.31$,
and there is no restriction
on the ratio between the coarse time step and fine
time step.
It is still possible to improve those estimations, by means of Theorem \ref{thm:conv-1}.
For example, one may obtain a smaller
convergence factor $\gamma$ by choosing a bigger $\alpha$ and
a smaller $\beta$, which might not affect the threshold $J_*=2$.
\end{remark}

\begin{remark}\label{rem:calahan}
In the proof of Propositions \eqref{prop:lobatto-2}-\eqref{prop:lobatto-4}, we employ the \textsl{L-stability} ($r(-\infty) = 0$) of the two-, three-, four-stage Lobbatto IIIC schemes. If the $r(-\infty) \in (0,1)$,
the analysis might be more technical, and the convergence might be
slow for small step ratio $J$; see e.g. Figure \ref{fig:ex1-b} for the Calahan scheme \eqref{eqn:Calahan}--\eqref{eqn:Calahan-p}.
However, Theorem \ref{thm:conv-1} guarantees the existence of the
threshold $J_*$ such that for any $J\ge J_*$ the convergence factor is close to $0.3$.
\end{remark}

\section{Numerical results}\label{sec:numerics}
In this section, we shall present some numerical examples to illustrate and complement our theoretical results.
To begin with, we use the one-dimensional diffusion models to show the sharpness of our convergence analysis in Sections \ref{sec:convergence} and \ref{sec:RK}.

\paragraph{Example 1. Linear Diffusion Models}
We consider the following initial-boundary value problem of parabolic equations
\begin{equation}\label{eqn:diffusion}
    \left \{
\begin{aligned}
&\partial_t u (x,t) - \partial_{xx} u(x,t) = f(x,t), &&\quad 0<t<T,\\
& u(x,t)=0,&&\quad x\in\partial\Omega,~0<t<T, \\
& u(x,0)=u^0(x),&&\quad x\in\Omega.
\end{aligned}
\right.
\end{equation}
where $\Omega=(0,\pi)$ and $T=1$. We consider the following two sets of problem data
\begin{itemize}
    \item[(a)] $u^0(x) = x^{5}(1-x)^{5}/(\pi/2)^{10}$
    and $f\equiv0$;
    \item[(b)] $u^0(x) = \chi_{(0,\frac\pi2)}(x)$ and $f=\cos(t) \sin(x)$, where $\chi_{(0,\frac\pi2)}(x)$ denotes the step function:
    \begin{equation*}
 \chi_{(0,\pi/2)}(x)  =
 \begin{cases}
 1, \quad x\in(0,\pi/2),\\
0, \quad \text{elsewise}.
 \end{cases}
\end{equation*}
\end{itemize}

In the computation, we divided the domain $\Omega$ into with $M$ equal subintervals of length $h=\pi/M$ and apply the Galerkin finite element with piesewise linear polynomials to discretize in space.
We examine the error between the parareal iterative solution $U_k^n$ and the exact time stepping solution $U^n$.
In our computation, we fixed 
spatial mesh size $h=\pi/1000$, and choose the initial guess
$U_0^n = u^0$ for all $n=0,\ldots,N$.

In example (a), the data is sufficiently smooth and compatible to the homogeneous Dirichlet boundary condition. In fact, it is easy to show show that $u^0 \in \text{Dom}((-\Delta)^{3+\epsilon})$ with $\epsilon\in(0,1/4)$  (see e.g. Lemma \cite[Lemma 3.1]{Thomee:2006}). For this case of regular data, Bal showed that the first several parareal iterations converge linearly with the rate $O(\Delta T)$; see cf. \cite{Bal:2005}. This is fully supported by the numerical results presented in Figure \ref{fig:ex1-a}, where we show the convergence of parareal algorithm for $2$- and $3$-stage Lobatto IIIC methods with
fixed $J=10$ and $\Delta T = 1/100$, $1/300$, $1/600$ (and correspondingly $\Delta t = 1/1000$, $1/3000$, $1/6000$).
We observe that the convergence of the first several iterations is faster for smaller coarse step size,
but the convergence then deteriorates for the later iterations.

\begin{figure} [tbhp]
\centering
\includegraphics[width=1\textwidth]{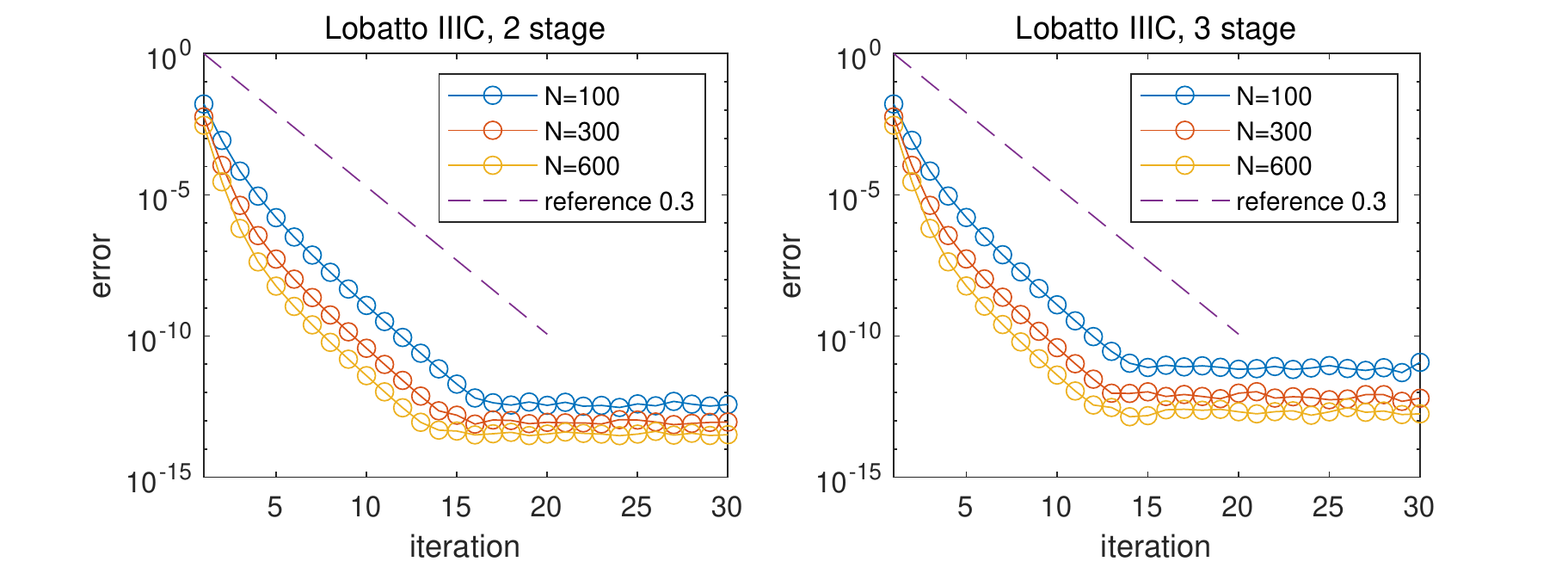}
\vskip-5pt   \caption{Example 1 (a): smooth data. Convergence of the parareal algorithm for $2$- and $3$-stage Lobatto IIIC methods with fixed mesh ratio $J=10$ and various coarse step sizes $1/N$, $N=100$, $300$, $600$.}
   \label{fig:ex1-a}
\end{figure}

In Figure \ref{fig:ex1-b}, we show the convergence of parareal algorithm for $2$-, $3$-, $4$-stage Lobatto IIIC methods solving
parabolic equation with nonsmooth initial data, i.e. Example 1 (b).
We fixed the fine step size $\Delta t = 1/3000$ and use different step ratios $J=2$, $3$ and $10$.
The numerical experiments clearly show that the parareal iterations converge linearly with convergence factor near $0.3$
for all $J\ge 2$. Meanwhile,
we observe that the convergence factor is independent of the ratio between coarse and find step sizes.
These phenomenon fully support our theoretical findings in  Propositions \ref{prop:lobatto-2}--\ref{prop:lobatto-4}.
Moreover, we test another time integrator,
called Calahan scheme \cite[eq. (1.9)]{Zlama:1974}, defined by
\begin{equation}\label{eqn:Calahan}
    r(-s) = 1-\frac{s}{1+bs} - \frac{\sqrt{3}}{6} \Big(\frac{s}{1+bs}  \Big)^2,\qquad \text{with}~~ b = \frac12\Big( 1 + \frac{\sqrt{3}}{3} \Big)
\end{equation}
and
\begin{equation}\label{eqn:Calahan-p}
\begin{aligned}
c_1 = \frac13,\qquad & p_1(-s) = \frac{(1/2 + \sqrt3) + (\sqrt3/2)s}{(1+bs)^2};\\
c_2 = \frac23,\qquad & p_2(-s) = \frac{(1/2 - \sqrt3) + (1/2 - \sqrt3/2)s}{(1+bs)^2}.
\end{aligned}
\end{equation}
The Butcher tableau is given by
\begin{equation}
\label{calahan}
\begin{tabular}{cc|c}
 $\frac16\sqrt{3}+\frac23$ &  $-\frac16\sqrt{3}-\frac13$  & $\frac13$\\[2pt]
 $-\frac16\sqrt{3}+\frac13$  & $\frac16\sqrt{3}+\frac13 $& $\frac23$ \\[2pt]
\hline
 $\frac12 $  & $\frac12$  & $\; $ \\
\end{tabular}
\ =:
\begin{tabular}{c|c}
 $\AA$ & $c$ \\[1pt]
\hline
 $b^\top$  & $\vphantom{\sum^{\sum^\sum}} $ \\
\end{tabular}
\end{equation}
It is easy to see that  $r(-s)$ is a decreasing function on $(0, \infty)$ and $r(-\infty)=1-\sqrt{3} \in (-1,0)$, so it is \textsl{A-stable}, but not \textsl{L-stable}.
Besides, the scheme is accurate of order $k=3$.
Therefore, the Calahan scheme satisfies Conditions (P1)--(P3).
Numerical results show that the converegnce of the corresponding parareal iterations is much slower than $0.3$ for small $J$.
This might be due to the fact that $|r(-\infty)| > 0$; see Remark \ref{rem:calahan}. However, for large $J$, the numerical results indicate a convergence rate close to $0.3$,
as predicted by Theorem \ref{thm:conv-2}.\vskip10pt

\begin{figure} [tbhp]
\centering
\includegraphics[width=0.9\textwidth]{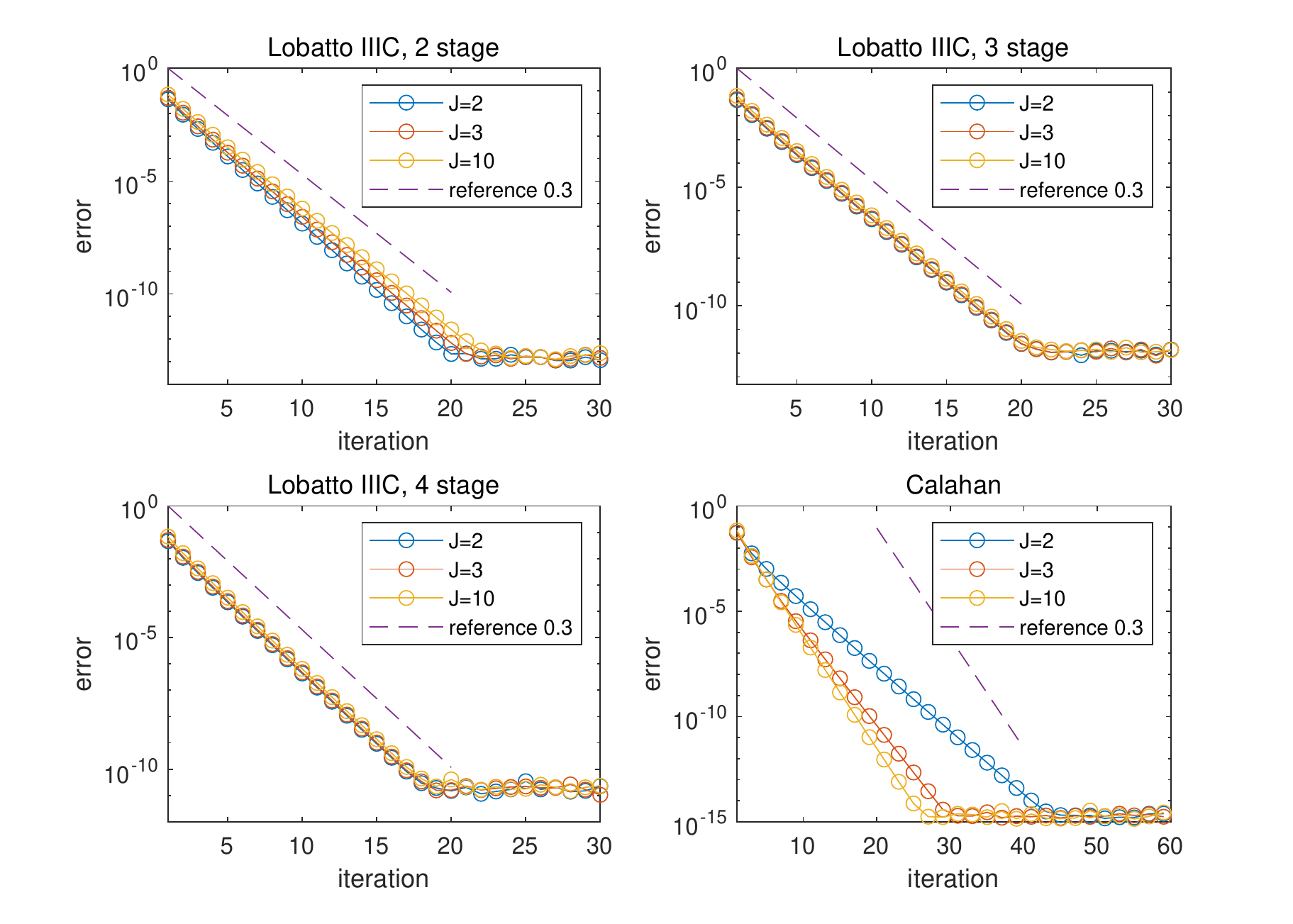}
   \caption{Example 1 (b): nonsmooth data. Convergence of the parareal algorithm for $2$-, $3$-, $4$-stage Lobatto IIIC methods and Calahan method with fixed fine step size $\Delta t = 1/3000$ and various ratios of coarse step size and fine step size.}
   \label{fig:ex1-b}
\end{figure}

\paragraph{Example 2. Semilinear Parabolic Equations}
In this part, we shall examine the convergence of parareal algorithm for solving the initial-boundary value problem of the semilinear parabolic equations
\begin{equation}\label{eqn:AC}
\left\{
\begin{aligned}
    \partial_t u   -  \partial_{xx} u &= \frac{1}{\epsilon^2}(u -u ^3)=:f(u ), \quad \text{for all}~~ x\in \Omega,t\in(0,T],\\
    \partial_x u(x,t)&=0, \qquad \qquad  \text{for all}~~ x\in\{0,\pi\},t\in(0,T],\\
    u(x,0)&=u^0(x),\quad\quad \text{for all}~~ x\in \Omega.
\end{aligned}
\right.
\end{equation}
The model \eqref{eqn:AC}, called Allen--Cahn equation, was originally introduced by Allen and Cahn in \cite{AC} to describe the motion of anti-phase boundaries in crystalline solids. In the context, $u$ represents the concentration of one of the two metallic components of the alloy and
the parameter $\epsilon$ involved in the nonlinear term  represents the width of interface. Recent decades, the Allen--Cahn equation has become one of basic phase-field equations, which has been widely applied to many complicated moving interface problems in materials science and fluid dynamics  \cite{anderson1998diffuse,chen2002phase,yue2004diffuse}.

In our numerical scheme, the coarse propagator  is the semli-implicit backward Euler scheme: for given $u^n$, look for $u^{n+1}$ such that for all $\phi\in H^1(0,\pi)$
$$ (u^{n+1}, \phi) + \Delta T(\partial_{x} u^{n+1},\partial_{x} \phi)= (u^n, \phi) + \Delta T (f(u^n), \phi),$$
which is uniquely solvable and first-order accurate, see e.g. \cite[Theorem 14.7]{Thomee:2006}.
Meanwhile, the fine propagator  is an arbitrary fully implicit high-order single step integrator (such as the Lobatto IIIC schemes or the fully implicit Calahan scheme): for given $u^n$, look for $u^{n+1}$ such that for all $\phi\in H^1(0,\pi)$
\begin{equation}\label{eqn:RK-fully}
\left\{
\begin{aligned}
    (u^{ni},\phi) &= (u^{n-1},\phi) + \Delta t \sum_{j=1}^m a_{ij} \Big(-(\partial_xu^{nj}, \partial_x \phi) +  (f(u^{nj}),\phi) \Big) ~~ \text{for} ~1\le i\le q, \\
    (u^{n},\phi) &= (u^{n-1},\phi) + \Delta t \sum_{i=1}^m b_i \Big(-(\partial_x u^{ni},\partial_x \phi) +  (f(u^{ni}),\phi)\Big),
\end{aligned}
\right.
\end{equation}
where the nonlinear system is uniquely solvable for sufficiently small step size, and we solve it
by use Newton's algorithm. Note that the fine propagator  is fully nonlinear and hence time consuming where the coarse propagator  is a linear scheme, so the application of parareal algorithm is able to significantly improve the efficiency.

\begin{figure} [tbhp]
\centering
\includegraphics[width=0.9\textwidth]{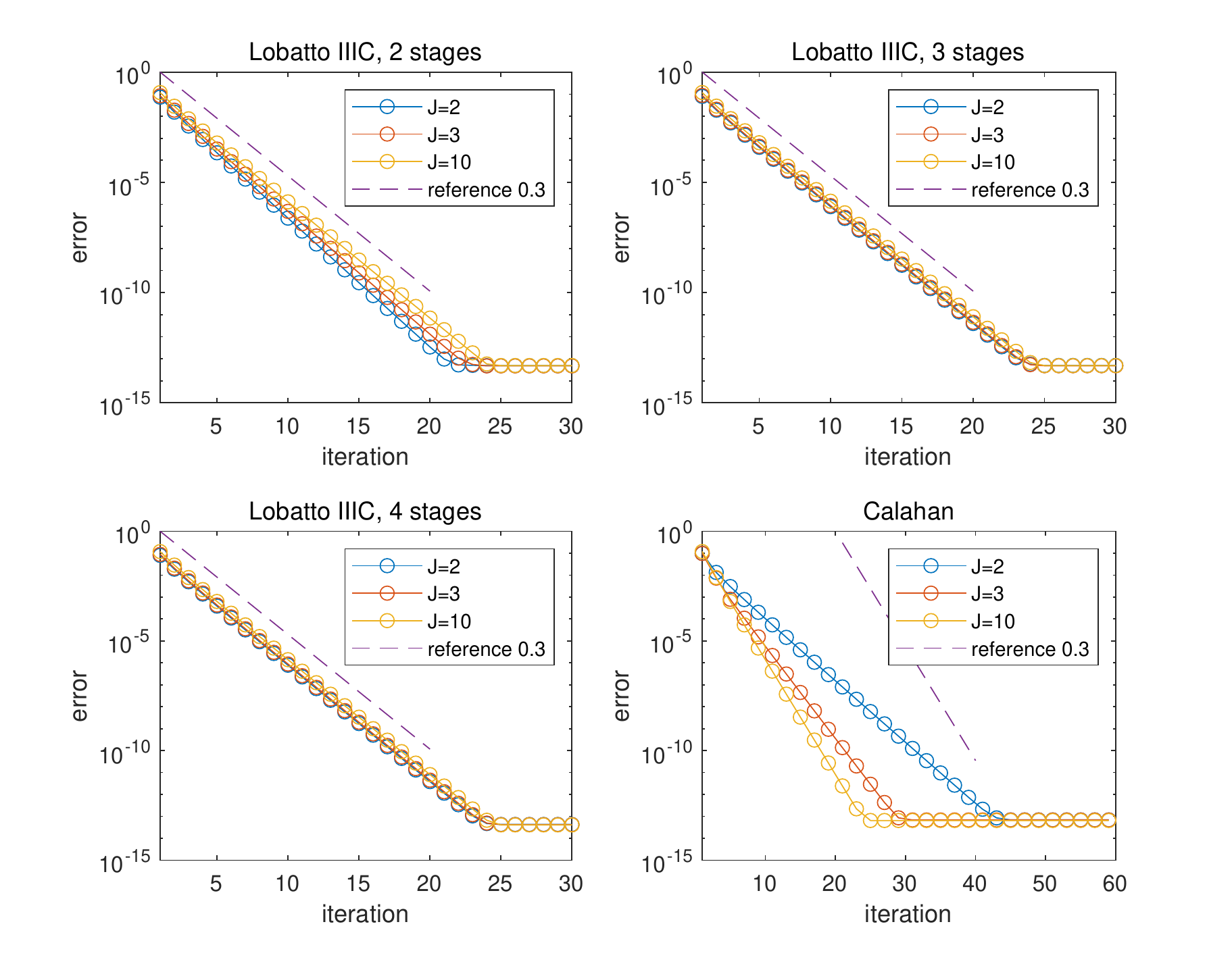}
   \caption{Example 2. Convergence of the parareal algorithm for $2$-, $3$-, $4$-stage Lobatto IIIC methods and Calahan method with fine step size $\Delta t = 1/600$ and various step ratios $J=2$, $3$, $10$.}
   \label{fig:ex2}
\end{figure}

In Figure \ref{fig:ex2}, we show the convergence of parareal algorithm for $2$-, $3$-, $4$-stage Lobatto IIIC methods and the Calahan method solving
the semilinear parabolic equation \eqref{eqn:AC}
with $\epsilon=1$ and $T=0.1$.
The fine step size is fixed and we examine the convergence for different step ratios.
Similar to the linear problem, for Lobatto IIIC methods,
the numerical experiments clearly show that the parareal iterations converge linearly with convergence factor near $0.3$
for all $J\ge 2$, while for the Calahan method the parareal iterations converge slowly for a small $J$.
The convergence analysis for the nonlinear problem warrants further investigation in our future studies. \vskip10pt


\section{Conclusion}
In this paper, we revisit the popular parareal algorithm for solving the linear parabolic equations, where the coarse propagator is fixed to the backward Euler method
and the fine propagator could be an arbitrarily high-order
single step integrator.
If the stability function of the fine propagator satisfies $|r(-\infty)|<1$,
we show that there must exist some critical threshold $J_*>0$ such
that the parareal solver converges as fast as Parareal-Euler with a convergence rate near $0.3$,
provided that the ratio between the coarse time step and fine
time step, named as $J$, satisfies $J \ge J_*$.
The convergence is robust even if the problem data is nonsmooth.  Moreover, we examine some popular high-order single step methods, e.g., two-, three- and four-stage Lobatto IIIC methods,
and verify a convergence factor
$0.31$ with the threshold $J_* = 2$.
The argument in the paper could be easily extended to the convection-diffusion equations.
Numerical experiments fully support our theoretical findings.

Some interesting questions are still open:
\begin{itemize}
    \item First of all, the argument in the paper, relies on the spectrum decomposition, which only works for the linear parabolic problem (with time-independent operator $A$). How to extend the argument to the case that $A$ is time-dependent and the nonlinear problem is still unclear. One possible approach is to combine the current analysis and the perturbation argument, as people did for the classical error analysis. Besides, in order to keep some important physical properties, like the maximum principle in the parabolic system, or the energy stable in the gradient flow system, many strategies are proposed in the development of numerical schemes, such as stabilization method \cite{ShenTangYang:2016, DuJuLiQiao:SIAMREV}, IEQ/SAV method \cite{ShenXuYang:SIAMRev, Yang:IEQ}, postprocessing method \cite{LiYangZhou:cut, Xia:post}. The parareal algorithm for those novel methods awaits theoretical studies.
    \item Moreover, in the current paper, we only consider the algorithm where the elliptic operator $A$ for the coarse and fine propagators keeps the same, which means we use the same spatial discretization. It is also interesting to analyze the parareal algorithm in the case that the spatial mesh sizes for coarse and fine propagators are different, i.e. $h_{\text{coarse}} \gg h_{\text{fine}}$. This is closely related to the space-time two grid method, whose convergence rate (robust with respect to the problem data) still awaits theoretical justification.
    \item Finally, we are interested in the parareal algorithm where the coarse or the fine propagators are stable linear multistep integrators. Both the development and the analysis of such algorithms are completely open, and the argument in current paper is not directly applicable. See some preliminary discussion about BDF2 scheme in \cite{Maday:BDF2}.
\end{itemize}

\vskip5pt

\bibliographystyle{siam}

\end{document}